\documentclass[11pt]{amsart}

\usepackage{amsmath}
\usepackage{amssymb}
\usepackage{amscd}

\usepackage{enumitem}
\usepackage{mathtools}
\usepackage{yhmath}

\usepackage[centertableaux]{ytableau}

\usepackage[colorlinks, linkcolor=purple, final]{hyperref}
\usepackage[doi=false, isbn=false, url=false]{biblatex}
\usepackage[tableposition=top]{caption}
\usepackage[mark=o]{dynkin-diagrams}
\usepackage{array}

\newcommand{\kk}{{\mathsf{k}}}

\newcommand{\RR}{{\mathbb{R}}}
\newcommand{\ZZ}{{\mathbb{Z}}}

\newcommand{\CE}{{\mathcal{E}}}
\newcommand{\CF}{{\mathcal{F}}}
\newcommand{\CG}{{\mathcal{G}}}

\newcommand{\CL}{{\mathcal{L}}}

\newcommand{\CO}{{\mathcal{O}}}

\newcommand{\CT}{{\mathcal{T}}}
\newcommand{\CU}{{\mathcal{U}}}

\newcommand{\Gr}{{\mathsf{Gr}}}
\newcommand{\OGr}{{\mathsf{OGr}}}
\newcommand{\IGr}{{\mathsf{IGr}}}
\newcommand{\LGr}{{\mathsf{LGr}}}

\newcommand{\GL}{{\mathsf{GL}}}

\newcommand{\SP}{{\mathsf{Sp}}}

\newcommand{\sfG}{\mathsf{G}}
\newcommand{\sfP}{\mathsf{P}}

\newcommand{\sfT}{\mathsf{T}}
\newcommand{\sfB}{\mathsf{B}}

\newcommand{\YD}{{\mathrm{Y}}}
\newcommand{\YDt}{{\mathrm{\tilde Y}}}
\newcommand{\rmP}{\mathrm{P}}
\newcommand{\BD}{\mathrm{B}}
\newcommand{\SO}{{\mathsf{SO}}}
\newcommand{\RB}{\mathrm{RB}}
\newcommand{\DB}{\mathrm{DB}}

\newcommand{\CUpU}{\CU^\perp\!/\,\CU}
\newcommand{\HH}{\mathrm{HH}}
\newcommand{\rmR}{\mathrm{R}}
\newcommand{\rmS}{\mathrm{S}}
\newcommand{\omegat}{\tilde{\omega}}

\DeclareMathOperator{\dht}{h}
\DeclareMathOperator{\dwd}{w}

\theoremstyle{plain}

\newtheorem{theo}{Theorem}[]

\newtheorem{cor}[theo]{Corollary}

\newtheorem{theorem}{Theorem}[section]
\newtheorem{conjecture}[theorem]{Conjecture}
\newtheorem{question}[theorem]{Question}
\newtheorem{lemma}[theorem]{Lemma}
\newtheorem{proposition}[theorem]{Proposition}

\theoremstyle{definition}

\newtheorem{definition}[theorem]{Definition}

\theoremstyle{remark}

\newtheorem{remark}[theorem]{Remark}
\newtheorem{example}[theorem]{Example}

\title{Hochschild Cohomology of Isotropic Grassmannians}
\author{Anton Fonarev}
\address{\sloppy
\parbox{0.95\textwidth}{
Algebraic Geometry Section, Steklov Mathematical Institute of Russian Academy of Sciences,\\
8 Gubkin str., Moscow 119991 Russia
\hfill
}\bigskip}
\email{avfonarev@mi-ras.ru}
\date{}
\thanks{This work was supported by the Russian Science Foundation grant No. 24-71-10092, \url{https://rscf.ru/en/project/24-71-10092/}.}

\begin{document}

\begin{abstract}
    We prove that nonspecial isotropic Grassmannians---that is,
    all isotropic Grassmannians which are neither
    (co)adjoint nor (co)minuscule, except $\mathsf{OGr}(n-1, 2n+1)$
    for $n\geq 4$---are not Hochschild global, thus establishing
    a conjecture by P.~Belmans and M.~Smirnov. As a corollary,
    we conclude that Bott vanishing fails for all these varieties. 
\end{abstract}

\maketitle

\section{Introduction}

\subsection{Hochschild cohomology}
Hochschild cohomology is an important homological invariant, which is closely related
to deformation theory.
Originally introduced by Gerhard Hochschild for associative
algebras over a field, it was later generalized to broader contexts.
In the present paper, we are interested in Hochschild cohomology
of isotropic Grassmannians. Though one can define Hochschild cohomology
of an algebraic variety in a more conceptual way, we are going
to take a shortcut, made possible by the so-called Hochschild--Kostant--Rosenberg
decomposition. From now on, we will only deal with smooth algebraic varieties over a field
of characteristic 0.

\begin{definition}
    Let $X$ be a smooth algebraic variety over a field of characteristic 0.
    Its $l$-th \emph{Hochschild cohomology} group is defined as
    \begin{equation}\label{eq:hkr}
        \HH^l(X) = \bigoplus_{i+j=l} H^i(X, \Lambda^j\CT_X),
    \end{equation}
    where $\CT_X$ denotes the tangent bundle of $X$.
\end{definition}

It follows immediately from the definition that if $X$ is of dimension $d$,
then $\HH^l(X) = 0$ for $l < 0$ and $l > 2d$. Indeed, negative exterior powers
and sheaf cohomology vanish by convention, and a theorem of Grothendieck
guarantees that sheaf cohomology of any sheaf of abelian groups vanishes
in degrees greater than the dimension of the space ($\Lambda^j\CT_X$ vanishes for $j>d$ since $\CT_X$ is a vector bundle of rank $d$).

Belmans and Smirnov gave the following definition in~\cite{Belmans2023}.
\begin{definition}
    A smooth algebraic variety $X$ is called \emph{Hochschild global} if
    \begin{equation*}
        H^i(X, \Lambda^j\CT_X) = 0 \text{ for all } j \text{ and } i > 0.
    \end{equation*}
    In other words, $X$ is Hochschild global if $\HH^l(X)=H^0(X, \Lambda^l\CT_X)$ for all $l\in\ZZ$.
\end{definition}

Hochschild globality of (co)adjoint and (co)minuscule isotropic Grassmannians
was established in~\cite{Belmans2023}, and the authors conjectured that
those isotropic Grassmannians which are neither (co)adjoint nor (co)minuscule
are not Hochschild global. In the present paper, we give a proof
of a weak form of this conjecture (see Section~\ref{ssec:main} for a precise statement).

\subsection{Isotropic Grassmannians}
Let $\kk$ be an algebraically closed field of characteristic 0, and let
$\sfG$ be a connected simply-connected simple algebraic group over $\kk$
of type $B$, $C$, or $D$. Fix a maximal torus $\sfT\subset \sfG$.
Denote by $\rmR$ the corresponding root system. Fix a subset
$\rmR^+\subset \rmR$ of positive roots, and denote by $\sfB$ the Borel
subgroup associated with the set of negative roots $-\rmR^+$.
With any subset $I\subset \rmS$ of the set of simple roots $\rmS\subset\rmR$
we associate the standard parabolic subgroup $\sfP_I$ containing $\sfB$.
In the present paper, we are interested in Hochschild cohomology of
\emph{isotropic Grassmannians},
i.e., varieties isomorphic to $\sfG/\sfP_i$, where $\sfP_i$ is the subgroup
associated with the set $S\setminus \{\alpha_i\}$, and $\alpha_i$ is
the $i$-th simple root (we use the Bourbaki labelling, see Table~\ref{tb:roots}).

\begin{table}[h]
    \caption{Bourbaki labelling of simple roots}
    \centering
    \begin{tabular}{m{1cm} m{4cm}} 
        type & Dynkin diagram \\
        \hline
        $B_n$ & \begin{tikzpicture}
        \dynkin[labels*={1,2,3, n-1, n}, scale=2, Bourbaki arrow] B{ooo...oo}
        \end{tikzpicture} \\
        $C_n$ & \begin{tikzpicture}
        \dynkin[labels*={1,2,3, n-1, n}, scale=2, Bourbaki arrow] C{ooo...oo}
        \end{tikzpicture} \\
        $D_n$ & \begin{tikzpicture}
        \dynkin[labels={1,2,n-3,n-2,n-1,n}, label directions={above,above,above,right,,}, scale=2, Bourbaki arrow] D{oo...oooo}
        \end{tikzpicture} \\
        \hline
    \end{tabular}
    \label{tb:roots}
\end{table}

Let $V$ be a finite-dimensional vector space over the field $\kk$ equipped with a non-degenerate bilinear
form $\omega\in V^*\otimes V^*$,
which we assume to be symmetric when $\dim V$ is odd and symmetric or antisymmetric when $\dim V$ is even.
If $\omega$ is antisymmetric, denote by $\IGr(k, V)$ the Grassmannian of $k$-dimensional isotropic subspaces in $V$.
If $\omega$ is symmetric, denote by $\OGr(k, V)$ the Grassmannian of $k$-dimensional isotropic subspaces in $V$.
If $\dim V=2n$ and $\omega$ is symmetric, then $\OGr(n, V)$ has two isomorphic connected components.
We pick one of them and denote it by $\OGr_+(n, V)$. When we do not want to specify $V$, we simply write its
dimension as the second parameter, e.g., $\IGr(k, 2n)$. Isotropic Grassmannians have the following geometric
description, see Table~\ref{tb:gr}.

\begin{table}[h]
    \caption{Isotropic Grassmannians}
    \centering
    \begin{tabular}{c l} 
        type & isotropic Grassmannians \\
        \hline
        $B_n$ & $\sfG/\sfP_k\simeq \OGr(k, 2n+1) \text{ for } k=1, \ldots, n $ \\
        $C_n$ & $\sfG/\sfP_k\simeq \IGr(k, 2n) \text{ for } k=1, \ldots, n $\\
        $D_n$ & $\sfG/\sfP_k\simeq \begin{cases}
        \OGr(k, 2n) &  \text{for } k=1, \ldots, n-2  \\
        \OGr_+(n, 2n) &  \text{for } k=n-1,n  \\
        \end{cases}$ \\
        \hline
    \end{tabular}
    \label{tb:gr}
\end{table}

The simple roots $\rmS$ form a basis of the space of weights.
We denote by $\alpha^\vee$ the coroot corresponding to the root $\alpha$.
Using the standard invariant scalar product $(-, -)$ on the space of weights of $\sfG$,
(see~\cite{Bourbaki}), we identify the space of weights of $\sfG$ and the space of coweights
in a way that $\alpha^\vee = \frac{2}{(\alpha, \alpha)}\alpha$ for all roots $\alpha$.

\begin{definition}
    A dominant weight $\lambda$ of $\sfG$ is called
    \begin{enumerate}
        \item \emph{minuscule} if $(\lambda, \alpha^\vee)\leq 1$ for all $\alpha\in\rmR^+$;
        \item \emph{cominuscule} if $(\alpha, \lambda^\vee)\leq 1$ for all $\alpha\in\rmR^+$;
        \item \emph{adjoint} if $\lambda$ is the highest weight of the adjoint representation of $\sfG$;
        \item \emph{coadjoint} if $\lambda$ is the highest short root. 
    \end{enumerate}
\end{definition}

\begin{definition}
    Let $X=\sfG/\sfP$ be an isotropic Grassmannian, and let $\omega$ be the highest weight associated with $\sfP_i$.
    The Grassmannian $X$ is called \emph{(co)minuscule} (resp. \emph{(co)adjoint}) if $\omega$
    is (co)minuscule (resp. (co)adjoint).
\end{definition}

\begin{definition}
    For the sake of the present paper we introduce the following terms.
    \begin{itemize}
        \item Call an isotropic Grassmannian \emph{special} if it is either
        (co)minuscule or (co)adjoint.
        \item Call the Grassmannians $\OGr(n-1, 2n+1)$ for $n\geq 4$ \emph{curious}.
        \item Call the Grassmannians which are neither special nor curious \emph{nonspecial}.
    \end{itemize}
\end{definition}

In Table~\ref{tb:good} we list all (co)minuscule and (co)adjoint isotropic Grassmannians.

\begin{table}[h]
    \caption{(Co)minuscule and (co)adjoint isotropic Grassmannians}
    \centering
    \begin{tabular}{c l l l l} 
        type & minuscule & cominuscule & adjoint & coadjoint \\
        \hline
        $B_n$ & $\OGr(n, 2n+1)$ & $\OGr(1, 2n+1)$ & $\OGr(2, 2n+1)$ & $\OGr(1, 2n+1)$ \\
        $C_n$ & $\IGr(1, 2n)$ & $\IGr(n, 2n)$ & $\IGr(1, 2n)$ & $\IGr(2, 2n)$ \\
        $D_n$ & $\OGr(1, 2n)$, $\OGr_+(n, 2n)$ & $\OGr(1, 2n)$, $\OGr_+(n, 2n)$
        & $\OGr(2, 2n)$ if $n\geq 4$ & $\OGr(2, 2n)$ if $n\geq 4$ \\
        \hline
    \end{tabular}
    \label{tb:good}
\end{table}

In view of our previous discussion, nonspecial isotropic Grassmannians are precisely the following varieties:
\begin{itemize}
    \item $\OGr(k, 2n+1)$ for all $3 \leq k \leq n-2$;
    \item $\IGr(k, 2n)$ for all $3 \leq k \leq n-1$;
    \item $\OGr(k, 2n)$ for all $3 \leq k \leq n-2$.
\end{itemize}

\subsection{Hochschild cohomology of isotropic Grassmannians}
Since the present work arose as an answer to a question from~\cite{Belmans2023},
we briefly recall its main results.
First, the authors prove the following theorem.

\begin{theorem}[{\cite[Theorem~A]{Belmans2023}}]\label{thm:bel-a}
    Let $X$ be a (co)minuscule or a (co)adjoint isotropic Grassmannian.
    Then $X$ is Hochschild global; that is,
    \begin{equation*}
        H^i(X, \Lambda^j\CT_X) = 0 \text{ for all } i>0 \text{ and all } j.
    \end{equation*}
\end{theorem}

\begin{remark}
    All classical Grassmannians $\Gr(k, V)$, associated with $\GL(V)$,
    are both minuscule and cominuscule (the two notions coincide in the
    simply-laced cases).
    Theorem~\ref{thm:bel-a} covers these varieties as well.
\end{remark}

\begin{remark}
    Belmans and Smirnov gave a representation theoretic description of
    $\HH^i(X)=\Gamma(X, \Lambda^i\CT_X)$ for special Grassmannians.
    We refer the interested reader to \cite[Theorems~B~and~C]{Belmans2023}
    for the corresponding results.
\end{remark}

The following conjecture formulates the problem that we address in the present
paper.

\begin{conjecture}[{\cite[Conjecture~F]{Belmans2023}}]\label{conj:bel}
    Let $X$ be an isotropic Grassmannian which is neither (co)minuscule
    nor (co)adjoint. Then $X$ is not Hochschild
    global. That is, for some $i > 0$ and some $j$ one has
    \begin{equation*}
        H^i(X, \Lambda^j\CT_X) \neq 0.
    \end{equation*}
\end{conjecture}

Belmans and Smirnov note in~\cite[Remark~1]{Belmans2023} that Conjecture~\ref{conj:bel}
is optimistic in the sense that it might not hold for orthogonal Grassmannians $\OGr(n-1, 2n+1)$
with $n\geq 4$,
which we call \emph{curious}.
Specifically, they analyzed a certain spectral sequence for $\OGr(3, 9)$
and saw potential vanishing of all higher cohomology of the exterior
powers of the tangent bundle. In the present paper we show that a weaker
form of Conjecture~\ref{conj:bel} holds. Namely, we establish the following.

\begin{conjecture}\label{conj:bel-weak}
    Let $X$ be a nonspecial isotropic Grassmannian. Then $X$ is not Hochschild
    global. That is, for some $i > 0$ and some $j$ one has
    \begin{equation*}
        H^i(X, \Lambda^j\CT_X) \neq 0.
    \end{equation*}
\end{conjecture}

In~\cite{Belmans2023}, the authors verify Conjecture~\ref{conj:bel-weak}
for nonspecial isotropic Grassmannians $\IGr(3, V)$,
where $V$ is a symplectic vector space of dimension $2n$, $n\geq 4$.
Precisely, they show the following.

\begin{proposition}[{\cite[Proposition~D]{Belmans2023}}]\label{prop:bel-3}
    Let $V$ be a symplectic vector space of dimension $2n$, where $n\geq 4$.
    Then
    \begin{equation}\label{eq:bel-3}
        H^1(\IGr(3, V), \Lambda^2\CT) \simeq V^{\langle \omega_4 \rangle},
    \end{equation}
    where $V^{\langle \omega_4 \rangle}$ is the irreducible representation of
    $\SP(V)\simeq \SP_{2n}$ of the highest weight $\omega_4$.
\end{proposition}

The following question remains open even for $\OGr(3, 9)$.

\begin{question}
    Let $X$ be a curious isotropic Grassmannian. That is, $X=\OGr(n-1, 2n+1)$
    for some $n\geq 4$. Is $X$ Hochschild global? 
\end{question}

\subsection{Main results}\label{ssec:main}
The main result of the paper gives an affirmative answer to Conjecture~\ref{conj:bel-weak}.
It immediately follows from the following two theorems. The first theorem
treats nonspecial isotropic Grassmannians in type $C$.

\begin{theo}\label{thm:c}
    Consider a symplectic Grassmannian $\IGr(k, V)$ which is nonspecial; that is, $3\leq k\leq n-1$, where $\dim V = 2n$.
    \begin{enumerate}[label=(\alph*)]
        \item\label{thm:c:1} For any $0\leq l\leq k-1$ one has
        \[
            \HH^l(\IGr(k, V)) = H^0(\IGr(k, V), \Lambda^l\CT).
        \]
        \item\label{thm:c:2} Let $j=\left\lfloor \frac{k}{2} \right\rfloor +1$. Then
        \[
            H^{k-j}(\IGr(k, V), \Lambda^j\CT)\neq 0.
        \]
        For instance, $\IGr(k, V)$ is not Hochschild global.
    \end{enumerate}
\end{theo}

Theorem~\ref{thm:c} is a combination of Proposition~\ref{pr:ig-hh0}, which
shows part~\ref{thm:c:1}, and Proposition~\ref{pr:igr-main}, which
shows part~\ref{thm:c:2} and, moreover, fully computes the corresponding cohomology group.

The second theorem treats nonspecial isotropic Grassmannians in types $B$ and $D$.

\begin{theo}\label{thm:bd}
    Consider an orthogonal Grassmannian $\OGr(k, V)$ which is nonspecial; that is, $3\leq k\leq n-2$, where $n=\lfloor \frac{N}{2} \rfloor$.
    \begin{enumerate}[label=(\alph*)]
        \item\label{thm:bd:1} For any $0\leq l\leq k-2$ one has
        \[
            \HH^l(\OGr(k, N)) = H^0(\OGr(k, 2n), \Lambda^l\CT).
        \]
        \item\label{thm:bd:2} One has
        \[
            H^{k-2}(\OGr(k, V), \Lambda^{k+2}\CT)\neq 0.
        \]
        For instance, $\OGr(k, V)$ is not Hochschild global.
    \end{enumerate}
\end{theo}

Theorem~\ref{thm:bd} is a combination of Proposition~\ref{pr:og-hh0}, which
shows part~\ref{thm:bd:1}, and Proposition~\ref{prop:ogr-k-2}, which
shows part~\ref{thm:bd:2}. Note that unlike the case of symplectic
nonspecial Grassmannians, for orthogonal nonspecial Grassmannians
we only show non-vanishing of the corresponding cohomology group.

\subsection{Bott vanishing}
As a corollary of Theorems~\ref{thm:c} and \ref{thm:bd}, we conclude
that Bott vanishing fails for nonspecial isotropic Grassmannians.
Specifically, let us say that a smooth projective variety $X$
of dimension $d$ satisfies \emph{Bott vanishing} if for
any ample line bundle $\CL$ one has
\[
    H^i(X, \Omega_X^j\otimes\CL) = 0
    \quad
    \text{for all }0\leq j \leq d \text{ and all } i>0.
\]
If $X$ is Fano---that is, $\omega_X=\det \Omega_X$ is anti-ample---then
the failure of being Hochschild global trivially implies
the failure of Bott vanishing since
$\Lambda^j\CT_X\simeq \Lambda^{d-j}\Omega_X\otimes \omega_X^{-1}$.
Since all rational homogeneous varieties, in particular, nonspecial
Grassmannians, are Fano, we get the following result.

\begin{cor}
    Bott vanishing fails for nonspecial isotropic Grassmannians.
\end{cor}

We refer the interested reader to~\cite[Section~5.3]{Belmans2023}
and~\cite{Belmans2025}
for further results on Bott vanishing for other rational homogeneous
varieties.

\subsection{Organization of the paper} The paper is organized as follows.
In Section~\ref{s:com} we collect all the preliminaries.
In Section~\ref{s:c} we treat nonspecial isotropic Grassmannians
in type $C$, and in Section~\ref{s:bd} we treat nonspecial isotropic
Grassmannians in types $B$ and $D$.

\section{Computational tools}\label{s:com}

\subsection{Young diagrams}
Throughout the paper we use the terms \emph{(Young) diagram} and \emph{partition}
interchangeably. We denote the set of partitions with $k$ parts by
\[
    \YD_k = \{\alpha \in \ZZ^k \mid \alpha_1\geq \alpha_2\geq \cdots\geq \alpha_k \geq 0 \}.
\]
We follow the convention under which $\lambda_1$ is the length of the first
row of the corresponding Young diagram. A \emph{strict} partition
$\lambda\in\YD_k$ is a partition satisfying
\[
    \lambda_1 > \lambda_2 > \cdots > \lambda_k > 0.
\]

For a Young diagram $\lambda\in\YD$, we denote by $\lambda^T$ its transpose:
\[
    {\lambda^T}\!\!_i = |\{j\geq 0 \mid \lambda_j \geq i\}|.
\]
The \emph{size} $\lambda$ is defined as $|\lambda| = \sum_i \lambda_i$.
For instance, $|\lambda| = |\lambda^T|$.
The \emph{width} of $\lambda$ is defined as the length of its first row, and
the \emph{height} of $\lambda$ is defined as the height of its first column.
We trivially have $\dwd(\lambda) = \lambda_1$, $\dht(\lambda)=\dwd(\lambda^T)$,
and $\dwd(\lambda) = \dht(\lambda^T)$.

Let $d(\lambda)$ be the size of the largest rectangle that can be inscribed
in $\lambda$, or, alternatively, the length of its diagonal:
\[
    d(\lambda)=\max\{i \mid \lambda_i\geq i\}.
\]
It is immediate that $d(\lambda)=d(\lambda^T)$.

Fix a nonnegative integer $d\geq 0$. There is a bijection between the set
of Young diagrams $\lambda$ with $d(\lambda)=d$ and pairs of strict
partitions with $d$ parts. Namely, for a box $(i, j)$ in $\lambda$
(which is defined by a pair of positive integers $i>0$ and $j>0$
such that $j\leq \lambda_i$),
define its \emph{arm length} as the number of boxes in the same row $\lambda$
which lie weakly east of $(i, j)$, and its \emph{leg length}
as the number of boxes in the same column of $\lambda$ which lie
weakly south of $(i, j)$. Remark that $(i, i)$ lies in $\lambda$ for all
$i=1,\ldots,d$. Let $a_i$ and $b_i$ denote the arm and leg
lengths of $(i, i)$, respectively. Precisely, $a_i=\lambda_i-i-1$, and
$b_i = {\lambda^T}\!\!_i-i-1$. One thus gets a pair of strict
partitions $a_1 > a_2 > \cdots a_d$ and $b_1 > b_2 > \cdots > b_d$,
and this association is the aforementioned bijection.
We denote the diagram corresponding to such a pair of strict partitions
by $(a_1, a_2, \ldots, a_d\mid b_1, b_2,\ldots, b_d)$. This notation
is commonly called the \emph{hook notation}.
For instance, $(t|1)$ corresponds to the partition $(t)$, while
$(1|t)$ corresponds to the partition $(1, 1, \ldots, 1)\in \YD_t$.
If $\lambda=(a_1, a_2, \ldots, a_d\mid b_1, b_2,\ldots, b_d)$, then
$\lambda^T=(b_1, b_2, \ldots, b_d\mid a_1, a_2, \ldots, a_d)$,
$d(\lambda)=d$, $\dwd(\lambda)=a_1$, $\dht(\lambda)=b_1$, and
$|\lambda| = \sum_i (a_i + b_i) - d$.

We call a diagram $\lambda = (a_1, a_2, \ldots, a_d\mid b_1, b_2,\ldots, b_d)$
$s$-\emph{balanced}
if for all $i=1,\ldots,d$
one has $a_i = b_i + s$. We denote the set of $s$-balanced diagrams by $\BD_s$.
There are two particular cases of balanced diagrams that we will encounter:
\begin{equation}\label{eq:balanced}
    \RB_q = \{\lambda\in \BD_{1} \mid |\lambda| = q\} \quad \text{and} \quad \DB_q = \{\lambda\in \BD_{-1} \mid |\lambda| = q\},
\end{equation}
which we call \emph{right} and \emph{down balanced}, respectively.
Since the transpose of an $s$-balanced diagram is $(-s)$-balanced,
we see that $\RB_q = \{\lambda^T \mid \lambda\in \DB_q\}$.
Right and down balanced diagrams are always of even size.
Here are the sets of right balanced diagrams of small size:
\[
    \ytableausetup{boxsize=.75em}
    \RB_2 = \left\{\ydiagram{2}\right\}, \quad
    \RB_4 = \left\{\ydiagram{3,1}\right\}, \quad
    \RB_6 = \left\{\ydiagram{3, 3},\ \ydiagram{4,1,1}\right\}.
\]
For every $\lambda\in \RB_{2r}$, one has $\dht(\lambda)\leq r$, and the equality
is achieved for exactly one such diagram: $(r+1\mid r)\in \YD_r$.
Similarly, for every $\lambda\in \DB_{2r}$, one has $\dht(\lambda)\leq r+1$,
and if $r>0$, then the equality is achieved for exactly one such diagram:
$(r \mid r+1)\in \YD_{r+1}$. Remark that when $r=0$,
the only element of $\DB_{0}$ is the empty diagram whose height equals $0$,
and our general upper bound for its height, which equals~$1$, is not attained.
This exceptional case will play
an important role when we discuss orthogonal Grassmannians.

\subsection{Weights}\label{ssec:weights}
We use the term \emph{weight} for an element of the dominant cone
$\rmP^+_{\sfG}$ of 
the weight lattice $\rmP_{\sfG}$
of a~(semi)simple algebraic group $\sfG$. When $\sfG=\SP(V)$ or $\sfG=\SO(V)$ is
of type $C_n$ or $B_n$/$D_n$, respectively, the set of partitions $\YD_n$ sits naturally in $\rmP^+_{\sfG}\subset \rmP_{\sfG}$.
Thus, we can speak of $\lambda\in\YD_n$ as of a weight of $\sfG$.
For any $\lambda\in \sfP^+_\sfG$,
we denote by $V^{\langle \beta \rangle}$ the corresponding irreducible representation.
Similarly, given a vector bundle $\CE$ on a scheme $X$ with the structure group $\sfG$, and a dominant weight $\lambda\in \sfP^+_\sfG$, we denote by $\CE^{\langle \lambda \rangle}$ the bundle associated with $\lambda$.
For instance, when $\sfG$ is of type $C_n$ or $B_n$/$D_n$, we can speak of $\CE^{\langle\lambda\rangle}$ for $\lambda\in\YD_n$.
In the following we will need a slightly more refined set than $\YD_n$,
so let us spell out the details.

In the following we denote by $\varepsilon_1, \varepsilon_2, \ldots, \varepsilon_n$ the standard basis of $\RR^n$,
and denote by $\rho$ the sum of the fundamental weights of $\sfG$.
We will need the following couple of definitions.

\begin{definition}
    A weight $\lambda\in \rmP_{\sfG}$ is called \emph{singular}
    if it is fixed by some nontrivial element of the Weyl group $W_{\sfG}$ of $\sfG$.
    A weight that is not singular will be called \emph{nonsingular}.
\end{definition}

\begin{definition}
    A weight $\lambda\in \rmP_{\sfG}$ is called \emph{strictly dominant} if
    $\lambda-\rho$ is dominant. Alternatively, a weight is strictly dominant
    if it is dominant and nonsingular.
\end{definition}

\subsubsection*{Type C}
Assume that $n\geq 2$, and let $\sfG = \SP_{2n}$. We identify
$\rmP_{\sfG}$ with $\ZZ^n\subset \RR^n$. The fundamental weights
are given by
\[
    \omega_1 = \varepsilon_1, \quad
    \omega_2 = \varepsilon_1 + \varepsilon_2, \quad \ldots, \quad
    \omega_n = \varepsilon_1 + \varepsilon_2 + \cdots + \varepsilon_n.   
\]
Their sum is
\[
    \rho = (n, n-1, \ldots, 2, 1)
\]
The dominant cone $\rmP^+_{\sfG}$ is identified with $\YD_n$:
\[
    \rmP^+_{\sfG} = \{ \lambda\in \ZZ^n \mid \lambda_1\geq \lambda_2\geq \cdots \geq \lambda_n\geq 0 \}.
\]
The Weyl group is the group of signed permutations---a semidirect product of the symmetric group $S_n$ and
the abelian group $(\ZZ/2\ZZ)^n$ which acts on $\rmP_{\sfG}$ by permutations
and sign changes. Thus, a weight $\lambda\in\rmP_{\sfG}=\ZZ^n$ is singular
if and only if the absolute values of two elements of $\lambda$ are the same
or one of the elements of $\lambda$ equals $0$.
We conclude that $\lambda$ is strictly dominant if and only if
\[
    \lambda_1 > \lambda_2 > \cdots > \lambda_n > 0.
\]

\subsubsection*{Type B}
Assume that $n\geq 2$, and let $\sfG = \SO_{2n+1}$. We identify
$\rmP_{\sfG}$ with $\ZZ^n + \tfrac{1}{2}\ZZ^n\subset \RR^n$. The fundamental weights
are given by
\[
    \omega_1 = \varepsilon_1, \quad
    \omega_2 = \varepsilon_1 + \varepsilon_2, \quad \ldots, \quad
    \omega_{n-1} = \varepsilon_1 + \varepsilon_2 + \cdots + \varepsilon_{n_1}, \quad   
    \omega_n = \frac{1}{2}(\varepsilon_1 + \varepsilon_2 + \cdots + \varepsilon_n).   
\]
Their sum is
\[
    \rho = \left(n-\tfrac{1}{2},\, n-\tfrac{3}{2},\, \ldots,\, \tfrac{3}{2},\, \tfrac{1}{2}\right)
\]
The dominant cone $\rmP^+_{\sfG}$ is identified with
\[
    \rmP^+_{\sfG} = \left\{ \lambda\in \ZZ^n + \tfrac{1}{2} \ZZ^n \mid \lambda_1\geq \lambda_2\geq \cdots \geq \lambda_n\geq 0 \right\}.
\]
Thus, $\YD_n\subset \rmP^+_{\sfG}$.

The Weyl group is again the group of signed permutations. Thus, a weight $\lambda\in\rmP_{\sfG}=\ZZ^n$ is singular
if and only if the absolute values of two elements of $\lambda$ are the same
or one of the elements of $\lambda$ equals $0$.
We conclude that $\lambda$ is strictly dominant if and only if
\[
    \lambda_1 > \lambda_2 > \cdots > \lambda_n > 0.
\]

\subsubsection*{Type D}
Assume that $n\geq 3$, and let $\sfG = \SO_{2n}$. We identify
$\rmP_{\sfG}$ with $\ZZ^n + \tfrac{1}{2}\ZZ^n\subset \RR^n$. The fundamental weights
are given by
\begin{gather*}
    \omega_1 = \varepsilon_1, \quad
    \omega_2 = \varepsilon_1 + \varepsilon_2, \quad \ldots, \quad
    \omega_{n-2} = \varepsilon_1 + \varepsilon_2 + \cdots + \varepsilon_{n_2}, \\
    \omega_{n-1} = \frac{1}{2}(\varepsilon_1 + \cdots + \varepsilon_{n-1} - \varepsilon_n),\quad
    \omega_n = \frac{1}{2}(\varepsilon_1 + \cdots + 
    \varepsilon_{n-1} + \varepsilon_n).   
\end{gather*}
Their sum is
\[
    \rho = \left(n-1,\, n-2,\, \ldots,\, 1,\, 0\right)
\]
The dominant cone $\rmP^+_{\sfG}$ is identified with
\[
    \rmP^+_{\sfG} = \left\{ \lambda\in \ZZ^n + \tfrac{1}{2} \ZZ^n \mid \lambda_1\geq \lambda_2\geq \cdots \geq |\lambda_n|\geq 0 \right\}.
\]
Thus, $\YD_n\subset \rmP^+_{\sfG}$.

The Weyl group is the group of signed permutations with an even number of sign
changes.
We conclude that $\lambda$ is strictly dominant if and only if
\[
    \lambda_1 > \lambda_2 > \cdots > |\lambda_n| \geq 0.
\]

We adopt the following slightly strange definition.

\begin{definition}
    Let $\sfG$ be of type $B_n$, $C_n$, or $D_n$.
    A weight $\lambda\in\rmP^+_{\sfG}$ is called \emph{integral} if $\lambda\in \rmP^+_{\sfG}\cap \ZZ^n$. We denote the set of integral dominant
    weights by $\YDt_n$.
\end{definition}

From the discussion above, we deduce that in types $C_n$ and $B_n$
one has $\YDt_n=\YD_n$. In type $D_n$, we have $\YD_n\subsetneq \YDt_n$, and
\[
    \YDt_n = \left\{ \lambda\in \ZZ^n \mid \lambda_1\geq \lambda_2\geq \cdots \geq |\lambda_n|\geq 0 \right\}.
\]

\subsection{Schur functor decompositions}
Given a Young diagram $\lambda$, we denote by $\Sigma^\lambda$ the corresponding Schur functor.
We follow the convention under which
$\Sigma^{(t\mid 1)} = S^t$ is the $t$-th symmetric power functor,
and $\Sigma^{(1\mid t)}=\Lambda^t$ is the $t$-th exterior power functor.
When $\lambda\in\rmP^+_{\GL_k}$, where we identify
\[
    \rmP^+_{\GL_k} = \{\lambda \in \ZZ^k \mid \lambda_1\geq \lambda_2\geq \cdots \geq \lambda_k\}.
\]
For a vector bundle $\CE$ of rank $k$, considered with the structure group $\GL_k$,
the bundle $\Sigma^\lambda\CE$ is isomorphic to $\CE^{\langle \lambda \rangle}$,
where $\lambda$ is considered as a dominant weight $\lambda\in\YD_k\subset \rmP^+_{\GL_k}$.

For the convenience of the reader, we present a few well known identities
for Schur functors applied to various bundles which hold in characteristic zero.
The first two are concerned
with taking exterior powers of extensions and tensor products.
In the following, whenever $\Sigma^\lambda$ is applied to a vector
bundle of rank smaller than $\dht(\lambda)$, the result is zero.

\begin{lemma}\label{lm:lambda-f}
    Let $0\to \CE\to \CF\to \CG\to$ be a short exact sequence of vector bundles
    on a scheme $X$, and let $q\geq 0$ be an integer.
    There is a filtration $0= \CF^{-1} \subset \CF^0\subset \cdots \subset \CF^q=\Lambda^q\CF$ on $\Lambda^q\CF$
    with the subqoutients
    \[
        \CF^j/\CF^{j-1}\simeq \Lambda^{q-j}\CE\otimes \Lambda^j\CG
    \]
    for all $j = 0, \dots, q$.
\end{lemma}

\begin{lemma}[{\cite[Corollary~2.3.3]{Weyman}}]\label{lm:lambda-t}
    Let $\CE$ and $\CF$ be vector bundles on a scheme $X$, and let $q\geq 0$
    be an integer. There is a direct sum decomposition
    \[
        \Lambda^q(\CE\otimes \CF) = \bigoplus_{|\lambda|=q}\Sigma^\lambda\CE\otimes \Sigma^{\lambda^T}\CF.
    \]
\end{lemma}

Next, we will need two basic examples of plethysm.
Recall that
$\RB_{2q}$ and $\DB_{2q}$ denote the sets of right and down balanced diagrams
of size $2q$, respectively.

\begin{lemma}[{\cite[Proposition~2.3.9]{Weyman}}]\label{lm:lambda-s2}
    Let $\CE$ be a vector bundle on a scheme $X$,
    and let $q\geq 0$ be an integer. There is a direct sum decomposition
    \[
        \Lambda^q (S^2\CE) = \bigoplus_{\mu\in \RB_{2q}}\Sigma^\mu \CE.
    \]
\end{lemma}

\begin{lemma}[{\cite[Proposition~2.3.9]{Weyman}}]\label{lm:lambda-l2}
    Let $\CE$ be a vector bundle on a scheme $X$,
    and let $q\geq 0$ be an integer. There is a direct sum decomposition
    \[
        \Lambda^q (\Lambda^2\CE) = \bigoplus_{\mu\in \DB_{2q}}\Sigma^\mu \CE.
    \]
\end{lemma}

The following is a simple consequence of the celebrated Littlewood--Richardson
rule, which allows decomposing products of Schur functors. We do not present
the rule here (nor its particular case know as Pieri's rule),
and refer the reader to~\cite{FultonYT} for details.
\begin{lemma}\label{lm:lr-h}
    Let $\alpha,\beta\in \YD_k$. For any rank $k$ bundle $\CE$ on a scheme $X$
    and any irreducible summand
    $\Sigma^\gamma\CE \subset \Sigma^\beta\CE \otimes \Sigma^\alpha\CE$
    one has
    \[
        \dht (\gamma) \leq \min(\dht (\alpha) + \dht (\beta),\,k).
    \]
\end{lemma}

\subsection{Spectral sequence of a filtered bundle}
Most of our cohomology computations will rely on the use of the spectral sequence
whose first page consists of the cohomology groups of the graded pieces
of a~filtration on a vector bundle. As usual, there are many ways to index
spectral sequences. We adopt the following convention. Let
$0= \CF^{-1} \subset \CF^0\subset \cdots \subset \CF^j=\CF$ be a filtration
on a vector bundle $\CF$ on a~scheme~$X$. There is a spectral sequence
$E^{q,i}_r$ with differential $d_r$ of degree $(-r, 1)$ and the first page
\begin{equation}\label{eq:ss}
    E^{q,i}_1 = H^i\big(X, \CF^q/\CF^{q-1}\big),
\end{equation}
which converges to cohomology of $X$; that is, there is a filtration on $H^i(X, \CF)$ with
the associated graded isomorphic to $\oplus_q E^{q, i}_{\infty}$.
We refer to~\cite[\href{https://stacks.math.columbia.edu/tag/0BKK}{Tag 0BKK}]{stacks-project}
for details.

\subsection{Borel--Bott--Weil and vanishing}
From now on,
let $X$ be a nonspecial isotropic Grassmannian of the form $\IGr(k, V)$ or $\OGr(k, V)$,
and let $n=\lfloor(\dim V) / 2\rfloor$.
We denote by $\sfG$ the corresponding simple group.
Whether the corresponding form $\omega$ on $V$ is symplectic
or symmetric, there is a closed embedding $X\hookrightarrow \Gr(k, V)$, and we denote
by $\CU$ the restriction of the rank $k$ tautological subbundle $\CU\subset V$.
The form $\omega$ induces an isomorphism $V^*\stackrel{\sim}{\to} V$,
which induces an isomorphism $(V/\CU)^*\stackrel{\sim}{\to} \CU^\perp$.
The isotropic condition implies that the embedding $\CU\hookrightarrow V$ factors
through $\CU^\perp$. The form $\omega$ induces a non-degenerate bilinear
form on $\CUpU$, which is symplectic (resp. symmetric).
Whether $\CUpU$ is symplectic or orthogonal, $\YD_{n-k}$ naturally sits in
the dominant cone of
its structure group (which is of type $B_{n-k}$, $C_{n-k}$, or $D_{n-k}$).
Recall that given $\lambda\in\YDt_{n-k}$, we denote
by $(\CUpU)^{\langle \lambda \rangle}$ the associated vector bundle
(see~Section~\ref{ssec:weights}).

\begin{example}
    If $\lambda = (2)$, then
    \[
        (\CUpU)^{\langle \lambda \rangle}\simeq \begin{cases}
            S^2(\CUpU) & \text{if } \CUpU \text{ is symplectic},\\
            S^2(\CUpU)/\CO & \text{if } \CUpU \text{ is orthogonal}.
        \end{cases}
    \]
    If $\lambda = (1,1)$, then
    \[
        (\CUpU)^{\langle \lambda \rangle}\simeq \begin{cases}
            \Lambda^2(\CUpU)/\CO & \text{if } \CUpU \text{ is symplectic},\\
            \Lambda^2(\CUpU) & \text{if } \CUpU \text{ is orthogonal}.
        \end{cases}
    \]
\end{example}

The following theorem is essentially a reformulation of the Borel--Bott--Weil theorem
for the cases that are of our interest. Addition and subtraction of sequences
that appear in it is done term-wise.

\begin{theorem}[\text{\cite[Corollaries~(4.3.4), (4.3.7), (4.3.9)]{Weyman}}]\label{thm:bbw}
    Assume that $k \leq n-2$.
    Let $\alpha\in\YD_k$, and $\beta\in\YDt_{n-k}$.
    Denote by $(\alpha, \beta)\in\ZZ^n$ the sequence
    $(\alpha, \beta)=(\alpha_1,\alpha_2,\ldots,\alpha_k,\beta_1, \beta_2,\ldots,\beta_{n-k})$,
    and let $\rho$ be the sum of fundamental weights, see Section~\ref{ssec:weights}. Consider the vector bundle $(\CUpU)^{\langle \beta \rangle}\otimes \Sigma^\alpha\CU^*$.
    \begin{enumerate}
        \item If $\rho+(\alpha, \beta)$ is a singular weight of $\sfG$,
        then
        \[
            H^\bullet(X,(\CUpU)^{\langle \beta \rangle}\otimes \Sigma^\alpha\CU^*)=0.
        \]
        \item If $\rho+(\alpha, \beta)$ is nonsingular weight of $\sfG$,
        let $\ell$ denote the number of inversions in it (considered as a sequence of distinct numbers), and let $\sigma$ be the unique permutation such that $\sigma(\rho+(\alpha, \beta))$ is strictly dominant.
        Then
        \[
            H^\bullet(X,(\CUpU)^{\langle \beta \rangle}\otimes \Sigma^\alpha\CU^*) = V^{\langle\sigma(\rho+(\alpha, \beta)) - \rho\rangle}[-\ell].
        \]
    \end{enumerate}
\end{theorem}

\begin{remark}\label{rm:eps}
    We will use the Borel--Bott--Weil theorem to prove some results
    on vanishing of cohomology groups of some vector bundles on $X$
    of the form $(\CUpU)^{\langle \beta \rangle}\otimes \Sigma^\alpha\CU^*$
    for $\alpha\in\YD_k$ and $\beta\in\YDt_{n-k}$. In order to make
    formulas uniform for all types, we introduce the following constant:
    \[
        \epsilon = \begin{cases}
            0 & \text{if }\sfG \text{ is of type } C, \\
            \tfrac{1}{2} & \text{if }\sfG \text{ is of type } B, \\
            1 & \text{if }\sfG \text{ is of type } D. \\
        \end{cases}
    \]
    Then, for $B_n$, $C_n$, and $D_n$, we can write
    \[
        \rho = (n-\epsilon, n-1-\epsilon, \ldots, 1-\epsilon).
    \]
\end{remark}

In the following two lemmas we assume that $1\leq k \leq n-2$.

\begin{lemma}\label{lm:iab00}
    Let $\alpha\in\YD_k$ and $\beta\in\YDt_{n-k}$ be such that $\dht(\alpha)+\beta_1 \leq k$.
    If $\beta \neq 0$, then
    \[
        H^\bullet(X,(\CUpU)^{\langle \beta \rangle}\otimes \Sigma^\alpha\CU^*)=0.
    \]
    If $\beta=0$, then
    \[
        H^{i}(X,(\CUpU)^{\langle \beta \rangle}\otimes \Sigma^\alpha\CU^*)=0\quad \text{for all}\quad i>0.
    \]
\end{lemma}
\begin{proof}
    Consider the~sequence $\gamma=\rho+(\alpha,\beta)$.
    If $\beta = 0$, then $(\alpha,\beta)$ is dominant. By Theorem~\ref{thm:bbw},
    the bundle $(\CUpU)^{\langle \beta \rangle}\otimes \Sigma^\alpha\CU^*\simeq \Sigma^\alpha\CU^*$
    has global sections equal to $V^{\langle \alpha \rangle}$ and no higher cohomology.
    Assume $\beta \neq 0$. In particular, $\beta_1\geq 1$.
    Since $\dht(\alpha)+\beta_1 \leq k$, it follows that $0\leq \dht(\alpha) \leq k-1$. Thus,
    \(
    (\gamma_{\dht(\alpha)+1}, \gamma_{\dht(\alpha)+2}, \ldots, \gamma_k) = 
    (\rho_{\dht(\alpha)+1}, \rho_{\dht(\alpha)+2}, \ldots, \rho_k) = 
    (n-\epsilon-\dht(\alpha), n-\dht(\alpha)-1-\epsilon, \ldots, n-k+1-\epsilon).
    \)
    Meanwhile, $\gamma_{k+1} = n-k-\epsilon + \beta_1$. Since $0 < \beta_1 \leq k-\dht(\alpha)$, we have
    $n-k-\epsilon < \gamma_{k+1} \leq n-\dht(\alpha)-\epsilon$. Thus,
    $\gamma_{k+1} = \gamma_i$ for some $\dht(\alpha)+1\leq i \leq k$,
    and the bundle $(\CUpU)^{\langle \beta \rangle}\otimes \Sigma^\alpha\CU^*$ is acyclic by Theorem~\ref{thm:bbw}.
\end{proof}

The next lemma guarantees vanishing of certain small cohomology groups of some equivariant vector bundles.

\begin{lemma}\label{lm:iab0}
    Let $\alpha\in\YD_k$ be such that $\dht(\alpha) \leq j$ for some integer $0\leq j<k$, and let $\beta\in\YDt_{n-k}$
    be arbitrary.
    Then
    \[
        H^i(X,(\CUpU)^{\langle \beta \rangle}\otimes \Sigma^\alpha\CU^*)=0
    \]
    for all $0 < i < k-j$.
\end{lemma}
\begin{proof}
    If $\beta_1\leq k-j$, then the assumptions of Lemma~\ref{lm:iab00}
    are satisfied, and there is no higher cohomology at all.
    Assume $\beta_1\geq k-j+1$.
    Consider the~sequence $\gamma=\rho+(\alpha,\beta)$. 
    If two elements in $\gamma$ are equal, then the bundle is acyclic.
    Otherwise, we have one nontrivial cohomology group in the degree
    equal to the number of inversions in $\gamma$.
    Since $\alpha_p=0$
    for $j+1\leq p \leq k$, we conclude that $\gamma_p=n+1-p-\epsilon$ for $j+1\leq p \leq k$.
    Since $\gamma_{k+1}=n-k-\epsilon+\beta_1 \geq n-j+1-\epsilon$, the pair $(p, k+1)$ is an
    inversion for all $j+1\leq p\leq k$, so the number of such inversions
    is at least $k-j$.
\end{proof}

\subsection{Restricted representations}
In the following we will need to decompose Schur functors applied
to symplectic and orthogonal bundles into direct sums
of bundles associated with representations of the corresponding
structure groups. The following is definitely known to specialists,
but we were not able to find a convenient reference.

\begin{lemma}\label{lm:res}
    Let $\CE$ be either a symplectic or orthogonal vector bundle
    of dimension $N\geq 2$ on a scheme $X$
    with the corresponding structure group $\sfG$ of type $C_n$, $B_n$, or $D_n$.
    Let $\lambda\in \YD_N$.
    \begin{enumerate}[label=(\arabic*)]
        \item\label{lm:res-i1} For any irreducible summand $\CE^{\langle \beta \rangle}$, $\beta\in \rmP^+_{\sfG}$,
        in the decomposition
        \begin{equation}\label{eq:res}
            \Sigma^\lambda\CE \simeq \bigoplus \CE^{\langle \beta \rangle},
        \end{equation}
        one has $\beta\in \YDt_n$ and
        \[
            \beta_1\leq \lambda_1, \quad \beta_2\leq \lambda_2, \quad \ldots,
            \quad \beta_{n-1}\leq \lambda_{n-1}, \quad |\beta_n|\leq \lambda_n.
        \]
        \item\label{lm:res-i2} If $\dht(\lambda)\leq n$, then $\lambda\in\YD_n\subset\YDt_n$
        and $\CE^{\langle \lambda \rangle}$ appears in decomposition~\eqref{eq:res}
        with multiplicity $1$.
    \end{enumerate}
\end{lemma}
\begin{proof}
    For simplicity, we only treat the case when $\sfG$ is of type $C_n$.
    Types $B_n$ and $D_n$ can be dealt with analogously.
    The bundle $\CE$ is symplectic of dimension $2n$.
    Consider the relative Lagrangian Grassmannian $p: \LGr_X(n, \CE)\to X$,
    and let $\CU\subset p^*\CE$ denote the universal subbundle of dimension $n$.
    One has a short exact sequence
    \begin{equation}\label{eq:res-univc}
        0\to \CU \to p^*\CE \to \CU^* \to 0.
    \end{equation}
    Note that $p_*p^*\Sigma^\lambda\CE \simeq \Sigma^\lambda\CE$, where the functors are taken in the derived sense.

    Short exact sequence~\eqref{eq:res-univc} induces a filtration on $\Sigma^\lambda \left(p^*\CE\right)\simeq p^*\Sigma^\lambda\CE$
    with the associated quotients isomorphic to
    \begin{equation}\label{eq:res-quotc}
        \bigoplus_{\mu\subseteq \lambda} \Sigma^\mu\CU\otimes \Sigma^{\lambda/\mu}\CU^*.
    \end{equation}
    Since everything is equivariant,
    a relative version of spectral sequence~\eqref{eq:ss} shows that
    $\Sigma^\lambda\CE$ is a subbundle of the bundle
    \[
        \bigoplus_{\mu\subseteq \lambda} R^0p_* \Sigma^\mu\CU\otimes \Sigma^{\lambda/\mu}\CU^*.
    \]

    Decompose $\Sigma^{\lambda/\mu}\CU^*$ as a direct sum (with possible
    multiplicities) $\Sigma^{\lambda/\mu}\CU^*\simeq \oplus \Sigma^\nu\CU^*$.
    Then $\nu\subseteq \lambda$. Finally,
    decompose $\Sigma^\mu\CU\otimes\Sigma^\nu\CU^*\simeq \oplus \Sigma^\alpha\CU^*$, where $\alpha\in\ZZ^n$,
    $\alpha_1\geq \alpha_2\geq\cdots\geq \alpha_n$, are weights of $\GL_n$.
    From the relative Borel--Weil theorem we know that
    $R^0\Sigma^\alpha\CU^*\simeq \CE^{\langle \alpha \rangle}$
    if $\alpha$ is $\sfG$-dominant, and $R^0\Sigma^\alpha\CU^*=0$ otherwise.
    Now, $\alpha\subseteq \lambda$ (which follows from the Littlewood--Richardson rule), and in the former case dominance simply means that $\alpha\in\YD_n=\YDt_n$. This proves~\ref{lm:res-i1}.

    Now,~\ref{lm:res-i2} follows from observing the $\Sigma^\lambda\CU^*$
    appears only once as a direct summand in~\eqref{eq:res-quotc},
    and $\CE^{\langle \lambda \rangle}$ does not appear in any
    $R^1p_*\Sigma^\alpha\CU^*$ (the first row of the spectral sequence).
\end{proof}

\section{Symplectic Grassmannians}\label{s:c}
\subsection{Setup}
Let $X$ be nonspecial isotropic of the form $\IGr(k, V)$, where
$V$ is a $2n$-dimensional vector space over $\kk$ equipped with a non-degenerate
symplectic form $\omega$. Recall that $3\leq k \leq n-1$.

The tangent bundle $\CT$ of $\IGr(k, V)$
is a (non-irreducible) equivariant vector bundle, which is an~extension of two irreducible
equivariant bundles
\begin{equation}\label{eq:tigr}
    0\to (\CUpU)\otimes \CU^* \to \CT \to S^2\CU^* \to 0,
\end{equation}
where $\CU$ is the tautological rank $k$ bundle, and $\CU^\perp=(V/\CU)^*$.
Applying Lemma~\ref{lm:lambda-f} to~\eqref{eq:tigr}, for any integer $j\geq 0$
we get a filtration on $\Lambda^j\CT_X$ with
the graded pieces of the form
\begin{equation}\label{eq:ilambda}
    \Lambda^p\big((\CUpU)\otimes \CU^*\big)\otimes \Lambda^q S^2\CU^*,
\end{equation}
where $q=0,\dots,j$, and $p=j-q$. We can further decompose
\begin{equation}\label{eq:ig-pq}
\begin{aligned}
    \Lambda^p\big((\CUpU)\otimes \CU^*\big)\otimes \Lambda^q S^2\CU^*
    & \simeq
    \bigoplus_{|\lambda|=p} \Sigma^\lambda(\CUpU)\otimes \Sigma^{\lambda^T}\CU^*\otimes\Lambda^q S^2\CU^* \\
    & \simeq
    \bigoplus_{\substack{|\lambda|=p,\\ \mu \in \RB_{2q}}} \Sigma^\lambda(\CUpU)\otimes \Sigma^{\lambda^T}\CU^*\otimes\Sigma^\mu\CU^*,
\end{aligned}
\end{equation}
where the first isomorphism in~\eqref{eq:ig-pq} follows from Lemma~\ref{lm:lambda-t},
while the second isomorphism in~\eqref{eq:ig-pq} follows from Lemma~\ref{lm:lambda-s2}.

\subsection{Vanishing}\label{ssec:igr-vanish}
We first show that up to degree $k-1$ Hochschild cohomology
of $\IGr(k, V)$ is Hochschild global. That is, we are going to prove the following.

\begin{proposition}\label{pr:ig-hh0}
    For any integer $0\leq l \leq k-1$ one has $\HH^l(\IGr(k, V))=H^0(\IGr(k, V),\Lambda^j\CT)$.
\end{proposition}

Remember that, according to our definition,
\[
    \HH^l(\IGr(k, V)) = \bigoplus_{j=0}^{l} H^{l-j}(\IGr(k, V), \Lambda^j\CT).
\]
Thus, in order to prove Proposition~\ref{pr:ig-hh0},
for any integer $0 < l \leq k-1$, and any $0\leq j < l$, we need to show
that $H^{l-j}(\IGr(k, V), \Lambda^j\CT)=0$. We show the latter
by studying cohomology
spectral sequence~\eqref{eq:ss} associated with filtration~\eqref{eq:ilambda}.

\begin{lemma}\label{lm:ig-ijk}
    Let $i > 0$ and $j\geq 0$ be such that $i+j<k$.
    Then all the terms in the $i$-th row $E^{\bullet, i}_1$ of cohomology
    spectral sequence~\eqref{eq:ss} associated with filtration~\eqref{eq:ilambda}
    vanish. In particular, one has $H^i(\IGr(k, V),\Lambda^j\CT) = 0$.
\end{lemma}
\begin{proof}
    Let $0\leq q \leq j$, and let $p=j-q$. Decomposition~\eqref{eq:ig-pq}
    implies that

    \[
        E^{q,i}_1 \simeq
        \bigoplus_{\substack{|\lambda|=p,\\ \mu \in \RB_{2q}}} H^i(\IGr(k, V),\Sigma^\lambda(\CUpU)\otimes \Sigma^{\lambda^T}\CU^*\otimes\Sigma^\mu\CU^*).
    \]
    Thus, it is enough to show that for any $\lambda$ with $|\lambda|=p$, and for any $\mu\in\mathrm{RB}_q$ one has
    \begin{equation*}
        H^i(\IGr(k, V),\Sigma^\lambda(\CUpU)\otimes \Sigma^{\lambda^T}\CU^*\otimes\Sigma^\mu\CU^*) = 0
    \end{equation*}
    if $0 < i < k-j$. Assume the latter.
    The bundle $\Sigma^\lambda(\CUpU)\otimes \Sigma^{\lambda^T}\CU^*\otimes\Sigma^\mu\CU^*$
    is a direct sum of the bundles $(\CUpU)^{\langle \beta \rangle}\otimes \Sigma^\alpha\CU^*$,
    where $(\CUpU)^{\langle \beta \rangle}$ and $\Sigma^\alpha\CU^*$ run
    over the irreducible summands of $\Sigma^\lambda(\CUpU)$ and
    $\Sigma^{\lambda^T}\CU^*\otimes\Sigma^\mu\CU^*$, respectively.
    From Lemma~\ref{lm:lr-h} we know that  $\dht (\alpha)\leq \dht (\lambda^T) + \dht (\mu)\leq p + q = j$.
    Thus, the conditions of Lemma~\ref{lm:iab0} are satisfied, and
    $H^i(\IGr(k, V),(\CUpU)^{\langle \beta \rangle}\otimes \Sigma^\alpha\CU^*)=0$,
    which completes the proof.
\end{proof}

\begin{proof}[Proof of Proposition~\ref{pr:ig-hh0}]
    Follows immediately from Lemma~\ref{lm:ig-ijk}.
\end{proof}

For small exterior powers of the tangent bundle we can show in a similar fashion
that there is no higher cohomology at all.
\begin{lemma}\label{lm:ig-jk}
    For any integer $j\geq 0$ such that $2j\leq k$,
    one has $H^i(\IGr(k, V),\Lambda^j\CT) = 0$ for all $i > 0$.
\end{lemma}
\begin{proof}
    We follow the proof of Lemma~\ref{lm:ig-ijk} and consider the same filtration.
    Let $(\CUpU)^{\langle \beta\rangle}\otimes \Sigma^\alpha\CU^*$ be an~irreducible
    summand in $\Sigma^\lambda(\CUpU)\otimes \Sigma^{\lambda^T}\CU^*\otimes\Sigma^\mu\CU^*$.
    By Lemma~\ref{lm:res}, $\beta\subset \lambda$, so $\dht(\beta)\leq \dht(\lambda) \leq p$,
    and $\dwd(\beta)\leq \dwd(\lambda) \leq p$.
    By Lemma~\ref{lm:lr-h}, $\dht (\alpha)\leq \dht (\lambda^T) + \dht (\mu)\leq p + q = j$.
    Assume that $H^i(\IGr(k, V), (\CUpU)^{\langle \beta\rangle}\otimes \Sigma^\alpha\CU^*)\neq 0$ for some $i>0$.
    From Lemma~\ref{lm:iab0}, we conclude that $\dht(\alpha)+\dwd(\beta) > k$.
    In particular, $j+p > k$, which implies $p > k-j$. Since $p \leq j$,
    one must have $2j > k$, which contradicts our assumptions.
\end{proof}

\subsection{Cohomology group \texorpdfstring{$\HH^k(\IGr(k, V))$}{k-th cohomology}}
We are now going to show that $\IGr(k, V)$ is not Hochschild global
by studying the~first cohomology group not covered by Proposition~\ref{pr:ig-hh0};
namely, we will look at
\[
    \HH^k(\IGr(k, V)) = \bigoplus_{i+j=k} H^i(\IGr(k, V), \Lambda^j\CT).
\]
From our point of view, it has the ``boring part'':
degree $k$ polyvector fields $H^0(\IGr(k, V), \Lambda^j\CT)$.
Meanwhile, we have seen in Lemma~\ref{lm:ig-jk} that if $2j\leq k$, then all higher cohomology
groups of $\Lambda^j\CT$ vanish. Recall that our assumption is that $3 \leq k \leq n-1$.
Recall that  $\omega_i$ denotes the $i$-th fundamental weight of $\SP(V)$
under Bourbaki labelling (see Section~\ref{ssec:weights}).

\begin{proposition}\label{pr:igr-main}
    Let $j = \left\lfloor \frac{k}{2} \right\rfloor+1$. Then
    \begin{equation}\label{eq:igr-main}
        H^{k-j}(\IGr(k, V), \Lambda^j\CT) \simeq V^{\langle (2j-k \mid k+1) \rangle} = 
        \begin{cases}
            V^{\langle \omega_{k+1} \rangle}, & \text{if } k \text{ is odd}, \\
            V^{\langle \omega_1+\omega_{k+1} \rangle}, & \text{if } k \text{ is even}.
        \end{cases}
    \end{equation}
\end{proposition}

\begin{remark}
    When $k=3$, Proposition~\ref{pr:igr-main} implies that
    \[
        H^1(\IGr(3, V), \Lambda^2\CT) \simeq V^{\langle \omega_4 \rangle}.
    \]
    Thus, we recover~\cite[Proposition~D]{Belmans2023}.
\end{remark}

It follows from Proposition~\ref{pr:igr-main} that $\IGr(k, V)$ is not Hochschild global.
Considering what we have shown in Section~\ref{ssec:igr-vanish},
$\HH^k(\IGr(k, V))$ is the smallest (in terms of degree)
non-global Hochschild cohomology group.
From Lemma~\ref{lm:ig-jk}, we know that $\Lambda^{\left\lfloor \frac{k}{2} \right\rfloor+1}\CT$
is the smallest exterior power that non-trivially contributes to $\HH^k(\IGr(k, V))$.
On the other end of the spectrum, we have $H^{0}(\IGr(k, V), \Lambda^k\CT)$---polyvector
fields of degree $k$. It might be interesting to fully compute
$\HH^k(\IGr(k, V))$; that is, to study $H^{k-j}(\IGr(k, V), \Lambda^j\CT)$
for $\left\lfloor \frac{k}{2} \right\rfloor+1 < j < k$.

\begin{question}\label{q:igr-main}
    Let $\left\lfloor \frac{k}{2} \right\rfloor < j < k$. Is it true that
    \begin{equation}\label{eq:igr-main-q}
        H^{k-j}(\IGr(k, V), \Lambda^j\CT) \simeq V^{\langle (2j-k \mid k+1) \rangle} = V^{\langle (2j-k-1)\omega_1+\omega_{k+1} \rangle}?
    \end{equation}
\end{question}

The rest of the section is devoted to our proof of Proposition~\ref{pr:igr-main}.
In order to compute the desired cohomology groups, we will use
filtration~\eqref{eq:ilambda} and spectral sequence~\eqref{eq:ss}
associated with~it.
We begin by working in greater generality
and fix positive integers $i>0$ and $j>\left\lfloor \frac{k}{2}\right\rfloor$ such that $i+j=k$. The terms of our spectral sequence are
\begin{equation}\label{eq:igr-ss}
    E^{q,i}_1 = H^i\big(\IGr(k, V),\, \Lambda^p\big((\CUpU)\otimes \CU^*\big)\otimes \Lambda^q S^2\CU^*\big),
\end{equation}
where $q=0,\ldots, j$, and $p=j-q$.

The following shows that there is only one non-zero term in the $i$-th row of~\eqref{eq:igr-ss}.

\begin{lemma}
    For $i$, $j$, $p$, and $q$ as above, one has
    \begin{equation*}
        H^i\big(\IGr(k, V),\, \Lambda^p\big((\CUpU)\otimes \CU^*\big)\otimes \Lambda^q S^2\CU^*\big) = \begin{cases}
            V^{\langle (2j-k \mid k+1) \rangle}, & \text{if}\ p=k-j+1, \\
            0, & \text{otherwise.}
        \end{cases}
    \end{equation*}
\end{lemma}
\begin{proof}
    Consider decomposition~\eqref{eq:ig-pq}
    \[
        \Lambda^p\big((\CUpU)\otimes \CU^*\big)\otimes \Lambda^q S^2\CU^*
        \simeq
        \bigoplus_{\substack{|\lambda|=p,\\ \mu \in \RB_{2q}}} \Sigma^\lambda(\CUpU)\otimes \Sigma^{\lambda^T}\CU^*\otimes\Sigma^\mu\CU^*,
    \]
    and let $\CE^{\beta,\alpha}=(\CUpU)^{\langle \beta \rangle}\otimes \Sigma^\alpha\CU^*$
    be an irreducible equivariant summand in $\Sigma^\lambda(\CUpU)\otimes \Sigma^{\lambda^T}\CU^*\otimes\Sigma^\mu\CU^*$.
    Observe that $\dht(\lambda^T)\leq p$ and $\dht(\mu)\leq q$.
    Thus, $\dht(\alpha)\leq \dht(\lambda^T) + \dht(\mu)\leq p+q = j$.
    If $\dht(\alpha) \leq j-1$, then from Lemma~\ref{lm:iab0} we conclude
    that $H^i(\IGr(k, V), \CE^{\beta,\alpha}) = 0$ since $i = k-j < k - (j-1)$.
    Thus, if $H^i(\IGr(k, V), \CE^{\beta,\alpha})$ does not vanish,
    then $\dht(\alpha)=j$. The latter means that $\dht(\lambda^T) = p$ and
    $\dht(\mu)=q$, which uniquely determines $\lambda$ and $\mu$:
    since $|\lambda|=p$, we conclude that $\lambda=(p)$, and, since $\mu\in \RB_{2q}$, we conclude that $\mu = (q+1\mid q)$. It follows from the Littlewood--Richardson/Pieri's rule that there is only one irreducible summand $\Sigma^\alpha\CU^*\subset\Sigma^{(p)^T}\CU^*\otimes\Sigma^{(q+1\,|\,q)}\CU^*$
    such that $\dht(\alpha)=p+q=j$: one has $\alpha = (q+1\,|\,j)$, and
    this summand appears with multiplicity one. Meanwhile, $\Sigma^{(p)}(\CUpU)=(\CUpU)^{\langle (p) \rangle}$ is an irreducible equivariant
    bundle. Thus, one has
    \[
        H^i\big(\IGr(k, V),\, \Lambda^p\big((\CUpU)\otimes \CU^*\big)\otimes \Lambda^q S^2\CU^*\big)
        \simeq H^i\big(\IGr(k, V),\, (\CUpU)^{\langle (p) \rangle}\otimes \Sigma^{(q+1\,|\,j)}\CU^*\big).
    \]
    We compute the latter using our simplified version of the Borel--Bott--Weil theorem (Theorem~\ref{thm:bbw}). Consider $\gamma=\rho+(\alpha, \beta)$:
    \begin{equation}
        \gamma = (n+q+1, n, n-1, \ldots, n-j+2, \underbrace{n-j, n-j-1, \ldots, n-k+1}_{k-j \text{ terms}}, n-k+p, n-k-1, \ldots, 2, 1).
    \end{equation}
    For instance, $\gamma_t=n+2-t$ for $t=2, \ldots, j$ and
    $\gamma_t = n+1-t$ for $t=j+1, \ldots, k$.
    Recall that we work under the assumption that $2j>k$. In particular,
    $k-j+1\leq j$. Consider the following cases for $0\leq p\leq j$.
    If $p=0$, then $\gamma$ is strictly decreasing, and the corresponding
    bundle has no higher cohomology. If $0 < p < k-j+1$, then $\gamma_{k+1}=n-k+p$
    equals $\gamma_{k+1-p}$, and the bundle is acyclic.
    Next, if~$k-j+1<p\leq j$, then $\gamma_{k+1}=\gamma_{k+2-p}$, and the bundle is
    acyclic once again. Finally, if $p=k-j+1$ (thus, $q=j-p=2j-k-1$), then all the terms in~$\gamma$
    are distinct, and there are
    $k-j=i$ inversions in~$\gamma$ given by the pairs of indices $(t, k+1)$, where $t=j+1, \ldots, k$.
    Once we put the elements of $\gamma$ in decreasing order and subtract
    $\rho$, we get precisely the weight $(q+1\mid k+1)=(2j-k\mid k+1)$.
\end{proof}

So far we have shown that the only nontrivial term in $E^{\bullet, i}_1$ is
\[
    E^{2j-k-1, i}_1 = V^{\langle (2j-k \mid k+1) \rangle}.
\]
Our spectral sequence~\eqref{eq:igr-ss} is equivariant, its $r$-th differential
us of degree $(-r, 1)$
so if we want to show that this term survives in $E^{\bullet, i}_\infty$,
it is enough to show that $V^{\langle (2j-k \mid k+1) \rangle}$
does not appear in $E^{s, i-1}_1$ for $s> 2j-k-1$ and in $E^{t, i+1}_1$
for $t < 2j-k-1$.
The former is rather easy. We show that $V^{\langle (2j-k \mid k+1) \rangle}$
does not appear in any term of the row $E^{\bullet, i-1}_1$.

\begin{lemma}
    The representation $V^{\langle (2j-k \mid k+1) \rangle}$ does not appear
    in $E^{\bullet, i-1}_1$.
\end{lemma}
\begin{proof}
    If $i-1>0$, then from Lemma~\ref{lm:ig-ijk}, we see that $E^{\bullet, i-1}_1=0$.
    Assume $i=1$, which implies $j=k-1$.
    Then $E^{q,0}_1 = H^0\big(\IGr(k, V),\, \Lambda^p\big((\CUpU)\otimes \CU^*\big)\otimes \Lambda^q S^2\CU^*\big)$, where
    $\Lambda^p\big((\CUpU)\otimes \CU^*\big)\otimes \Lambda^q S^2\CU^*$
    is a~direct sum of vector bundles of the form
    $\CE^{\alpha, \beta}=(\CUpU)^{\langle \beta \rangle}\otimes \Sigma^\alpha\CU^*$
    with $\dht(\alpha)\leq p+q=j=k-1$.
    From the Borel--Bott--Weil theorem we know that the bundle $\CE^{\alpha, \beta}$
    contributes to global sections if and only if all the terms
    in~$\gamma=\rho+(\alpha, \beta)$ are strictly decreasing.
    Assume the latter.
    Since $\dht(\alpha)\leq k-1$, one has
    $n-k+1=\gamma_k > \gamma_{k+1} = n-k+\beta_1$.
    Since $\beta_1\geq 0$, we conclude that
    $\beta_1=0$. Thus, $\beta=0$. Finally, remark that $H^0(\IGr(k, V), \CE^{\alpha, 0}) = V^{\langle \alpha \rangle}$, but $V^{\langle \alpha \rangle}$ can not be isomorphic
    to $V^{\langle (2j-k \mid k+1) \rangle}$ since $\dht(\alpha) \leq k-1$ and $\dht\!\big((2j-k \mid k+1)\big) = k+1$.

\end{proof}

It is tempting to conjecture that the representation
$V^{\langle (2j-k \mid k+1) \rangle}$ does not appear in the $(i+1)$-st
row of the spectral sequence, $E^{\bullet, i+1}_1$.
In Lemma~\ref{lm:igr-top} we show that it is indeed true when
$j=\lfloor\frac{k}{2}\rfloor+1$. However, it is not true
for $\lfloor\frac{k}{2}\rfloor+1 < j < k$,
see Remark~\ref{rm:igr-top}.

\begin{lemma}\label{lm:igr-top}
    Assume $j=\lfloor\frac{k}{2}\rfloor+1$.
    The representation $V^{\langle (2j-k \mid k+1) \rangle}$ does not appear
    in $E^{q, i+1}_1$ for $q < 2j-k-1$. In particular,
    $E^{2j-k-1, i}_1 = E^{2j-k-1, i}_\infty$.
\end{lemma}
\begin{proof}
    The second statement follows from the spectral sequence being equivariant.
    Let us show the first one. Consider two cases, based on the parity of $k$.
    Assume that $k=2t-1$ for some $t\geq 2$.
    Then $j=t$, and $2j-k-1=0$. Since all the terms $E^{q,\bullet}_\bullet$
    are zero for $q<0$, the statement follows.
    Now, consider the case $k=2t$ for some $t\geq 2$. Then $j=t+1$,
    and $2j-k-1=1$. Thus, we only need to show that $V^{\langle (2 \mid k+1) \rangle}$ does not appear in
    \[
        E^{0,i+1}_1 = H^t\big(\IGr(k, V), \Lambda^{t+1}((\CUpU)\otimes \CU^*)\big) \simeq
        \bigoplus_{|\lambda|=t+1}
        H^t\big(\IGr(k, V), \Sigma^\lambda(\CUpU)\otimes \Sigma^{\lambda^T}\CU^*\big).
    \]
    Let $\CE^{\beta}=(\CUpU)^{\langle \beta \rangle}\otimes \Sigma^{\lambda^T}\CU^*$ be an irreducible summand.
    If $\dwd(\lambda)=t+1$, then $\lambda=(t+1)$ and
    $\Sigma^\lambda(\CUpU)\otimes \Sigma^{\lambda^T}\CU^* = \Sigma^{\langle(t+1)\rangle}(\CUpU)\otimes \Sigma^{(1|t+1)}\CU^*$. The cohomology of the
    latter bundle vanishes by the Borel--Bott--Weil theorem.
    If $\dwd(\lambda)\leq t$, then
    $\beta_1+\dht(\lambda^T)\leq 2\dwd(\lambda) \leq k$;
    so, Lemma~\ref{lm:iab00} implies the vanishing.
\end{proof}

\begin{remark}\label{rm:igr-top}
    Assume that $\lfloor\frac{k}{2}\rfloor+1 < j < k$.
    As before, let $i=k-j$, let $q=2j-k-2$, and let $p=j-q$. We claim that there is an irreducible summand
    in $E^{2j-k-2, i+1}_1$ isomorphic to $E^{2j-k-1, i}_1 = V^{\langle (2j-k \mid k+1) \rangle}$.
    In particular, the differential $d^{2j-k-1, i}_1$ might kill
    $E^{2j-k-1, i}_1$.
    Consider the $q$-th term of the filtration~\eqref{eq:ilambda}.
    It~contains a summand isomorphic to $S^{p}(\CUpU)\otimes \Lambda^p \CU^*\otimes \Sigma^{(q+1\,|\,q)}\CU^*$.
    Recall that $S^{p}(\CUpU)=(\CUpU)^{\langle (p) \rangle}$.
    Meanwhile,
    $\Lambda^p \CU^*\otimes \Sigma^{(q+1\,|\,q)}\CU^*$
    contains an irreducible summand isomorphic to
    $\Sigma^{(q+2\,|\,q+p-1)}\CU^*$.
    A simple Borel--Bott--Weil computation shows that
    \[
        H^{i+1}\big(\IGr(k, V),\, (\CUpU)^{\langle (p) \rangle}\otimes \Sigma^{(q+2\,|\,j-1)}\CU^*\big) \simeq V^{\langle (2j-k \mid k+1) \rangle}.
    \]
\end{remark}

\section{Orthogonal Grassmannians}\label{s:bd}
\subsection{Setup}
Let $X$ be nonspecial isotropic of the form $\OGr(k, V)$, where
$V$ is a vector space over $\kk$ equipped with a non-degenerate
symmetric form $\omega$. Put $n=\lfloor \frac{\dim V}{2}\rfloor$.
Recall $3\leq k\leq n-2$. 

The tangent bundle $\CT$ of $\OGr(k, V)$
is a (non-irreducible) equivariant vector bundle, which is an~extension of two irreducible
equivariant bundles
\begin{equation}\label{eq:togr}
    0\to (\CUpU)\otimes \CU^* \to \CT \to \Lambda^2\CU^* \to 0,
\end{equation}
where $\CU$ is the tautological rank $k$ bundle, and $\CU^\perp\simeq(V/\CU)^*$.
The difference between the orthogonal and the symplectic case is that
the term $S^2\CU^*$ gets replaced by $\Lambda^2\CU^*$.
Applying Lemma~\ref{lm:lambda-f} to~\eqref{eq:tigr}, for any integer $j\geq 0$
we get a filtration on $\Lambda^j\CT_X$ with
the graded pieces of the form
\begin{equation}\label{eq:olambda}
    \Lambda^p\big((\CUpU)\otimes \CU^*\big)\otimes \Lambda^q \Lambda^2\CU^*,
\end{equation}
where $q=0,\dots,j$, and $p=j-q$. We can further decompose
\begin{equation}\label{eq:og-pq}
\begin{aligned}
    \Lambda^p\big((\CUpU)\otimes \CU^*\big)\otimes \Lambda^q \Lambda^2\CU^*
    & \simeq
    \bigoplus_{|\lambda|=p} \Sigma^\lambda(\CUpU)\otimes \Sigma^{\lambda^T}\CU^*\otimes\Lambda^q \Lambda^2\CU^* \\
    & \simeq
    \bigoplus_{\substack{|\lambda|=p,\\ \mu \in \DB_{2q}}} \Sigma^\lambda(\CUpU)\otimes \Sigma^{\lambda^T}\CU^*\otimes\Sigma^\mu\CU^*,
\end{aligned}
\end{equation}
where the first isomorphism in~\eqref{eq:og-pq} follows from Lemma~\ref{lm:lambda-t},
while the second isomorphism in~\eqref{eq:og-pq} follows from Lemma~\ref{lm:lambda-l2}.

\subsection{Vanishing} We first show that up to degree $k-2$ Hochschild cohomology
of $\OGr(k, V)$ is Hochschild global. That is, we are going to prove the following.

\begin{proposition}\label{pr:og-hh0}
    For any integer $0\leq l \leq k-2$ one has $\HH^l(\OGr(k, V))=H^0(\OGr(k, V),\Lambda^j\CT)$.
\end{proposition}

Recall that
\[
    \HH^l(\OGr(k, V)) = \bigoplus_{j=0}^{l} H^{l-j}(\OGr(k, V), \Lambda^j\CT).
\]
In order to prove Proposition~\ref{pr:ig-hh0},
for any integer $0\leq l \leq k-2$ and any $0\leq j < l$ we show
that $H^{l-j}(\IGr(k, V), \Lambda^j\CT)=0$
by studying cohomology
spectral sequence~\eqref{eq:ss} for the filtration with the associated subquotients~\eqref{eq:olambda}.

\begin{lemma}\label{lm:og-ijk}
    Consider integers $i > 0$ and $j\geq 0$ such that $i+j\leq k-2$.
    Then all the terms in the $i$-th row $E^{\bullet, i}_1$ of
    spectral sequence~\eqref{eq:ss} associated with the filtration with subquotients~\eqref{eq:olambda}
    vanish. In particular, one has $H^i(\IGr(k, V),\Lambda^j\CT) = 0$.
\end{lemma}
\begin{proof}
    Consider decomposition~\eqref{eq:og-pq}.
    Let us show that for any $p,q\geq 0$ such that $p+q=j$,
    for any $\lambda$ with $|\lambda|=p$, and for any $\mu\in\mathrm{DB}_q$ one has
    \begin{equation*}
        H^i(\OGr(k, V),\Sigma^\lambda(\CUpU)\otimes \Sigma^{\lambda^T}\CU^*\otimes\Sigma^\mu\CU^*) = 0
    \end{equation*}
    if $0 < i < k-j-1$.
    Let $(\CUpU)^{\langle \beta \rangle}\subset \Sigma^\lambda(\CUpU)$
    and $\Sigma^\alpha\CU^*\subset \Sigma^{\lambda^T}\CU^*\otimes\Sigma^\mu\CU^*$
    be irreducible summands. 
    Then $(\CUpU)^{\langle \beta \rangle}\otimes \Sigma^\alpha\CU^*$ is an irreducible
    summand in $\Sigma^\lambda(\CUpU)\otimes \Sigma^{\lambda^T}\CU^*\otimes\Sigma^\mu\CU^*$.
    From Lemma~\ref{lm:lr-h} we know that  $\dht (\alpha)\leq \dht (\lambda^T) + \dht (\mu)\leq p + (q+1) = j+1$.
    Thus, the conditions of Lemma~\ref{lm:iab0} are satisfied, and
    $H^i(\OGr(k, V),(\CUpU)^{\langle \beta \rangle}\otimes \Sigma^\alpha\CU^*)=0$,
    which completes the proof.
\end{proof}

\begin{proof}[Proof of Proposition~\ref{pr:og-hh0}]
    Follows immediately from Lemma~\ref{lm:og-ijk}.
\end{proof}

The following statement is an analogue, in the orthogonal case, of Lemma~\ref{lm:ig-jk}.
\begin{lemma}\label{lm:og-jk}
    For any integer $j\geq 0$ such that $2j < k$,
    one has $H^i(\IGr(k, V),\Lambda^j\CT) = 0$ for all $i > 0$.
\end{lemma}
\begin{proof}
    We follow the proof of Lemma~\ref{lm:ig-ijk} and consider the same filtration.
    Let $(\CUpU)^{\langle \beta\rangle}\otimes \Sigma^\alpha\CU^*$ be an~irreducible
    summand in $\Sigma^\lambda(\CUpU)\otimes \Sigma^{\lambda^T}\CU^*\otimes\Sigma^\mu\CU^*$.
    By Lemma~\ref{lm:res}, $\beta_1\leq \lambda_1\leq p$.
    Since $\mu\in\DB_q$, $\dht(\mu)\leq q+1$. Thus,
    $\dht (\alpha)\leq \dht (\lambda^T) + \dht (\mu)\leq p + q + 1 = j+1$.
    If $H^i(\IGr(k, V), (\CUpU)^{\langle \beta\rangle}\otimes \Sigma^\alpha\CU^*)\neq 0$ for some $i>0$,
    then it follows from Lemma~\ref{lm:iab0} that $\dht(\alpha)+\beta_1 > k$.
    In particular, $j+1+p > k$, which implies $p \geq k-j$. Since $p \leq j$,
    one must have $2j \geq k$, which contradicts our assumptions.
\end{proof}

\subsection{Cohomology group \texorpdfstring{$\HH^{2k}(\OGr(k, V))$}{2k-th cohomology}}

We are going to show that nonspecial orthogonal Grassmannians fail to be Hochschild global
by studying their $2k$-th Hochschild cohomology groups.

In order to formulate the following results in a uniform manner, we introduce the following
notation.
In type $B_n$ denote
\[
    \omegat_i = \begin{cases}
        \omega_i & \text{if } i=0,\ldots, n-1, \\
        2\omega_n & \text{if } i=n.
    \end{cases}
\]
In type $D_n$ denote
\[
    \omegat_i = \begin{cases}
        \omega_i & \text{if } i=0,\ldots, n-2, \\
        \omega_{n-1} + \omega_n & \text{if } i=n-1, \\
        2\omega_n & \text{if } i=n.
    \end{cases}
\]
Considering our embedding of $\YD_n$ in the dominant cone $\rmP^+_{\sfG}$,
the weight $\omegat_i$ simply equals
\[ \omegat_i = (\underbrace{1, 1, \ldots, 1}_{k \text{ times}}, 0, \ldots, 0).\]

\begin{proposition}\label{prop:ogr-k-2}
    Let $X=\OGr(k, V)$ be nonspecial. Then $V^{\langle 2\omega_{k+2} \rangle}\subset H^{k-2}(X, \Lambda^{k+2}\CT)$.
\end{proposition}

Remark that the very statement of Proposition~\ref{prop:ogr-k-2} makes no sense for the curious Grassmannian
$\OGr(n-1, 2n+1)$, since $k=n-1$ implies $k+2>n$.

The following lemma is the key statement in our proof of Proposition~\ref{prop:ogr-k-2}.
Recall that $\epsilon$ was introduced in Remark~\ref{rm:eps}.

\begin{lemma}\label{lm:ogr-k-2}
    Let $X=\OGr(k, V)$ be nonspecial, let $a\geq b\geq 0$ be integers,
    and let $\alpha\in\YD_k$. Then
    \[
        H^l(X, (\CUpU)^{\langle a, b \rangle}\otimes \Sigma^\alpha\CU^*) =
        V^{\langle 2\omega_{k+2} \rangle}
    \]
    if and only if $a \leq k+2$, $b \geq 2$, $\alpha^T = (k+2-b, k+2-a)$, and $l = a+b-4$.
\end{lemma}
\begin{proof}
    Assume $H^l(X, (\CUpU)^{\langle a, b \rangle}\otimes \Sigma^\alpha\CU^*) = V^{\langle 2\omega_{k+2} \rangle}$.
    By the Borel--Bott--Weil theorem, the elements of the sequence
    \begin{equation}\label{eq:ogr-k-2}
        (n-\epsilon+\alpha_1, n-1-\epsilon+\alpha_2, \ldots, n-k+1-\epsilon+\alpha_1, n-k-\epsilon+a, n-k-1-\epsilon+b)
    \end{equation}
    are all distinct and, once put in decreasing order, form the sequence
    \[
        (n+2-\epsilon, n+1-\epsilon, \ldots, n-k+1-\epsilon).
    \]
    Thus, $n-k-\epsilon+a\leq n+2-\epsilon$, which implies $a\leq k+2$, and $n-k-1-\epsilon+b\geq n-k+1-\epsilon$,
    which implies $b\geq 2$.
    Since the first $k$ elements in~\eqref{eq:ogr-k-2} are strictly decreasing,
    they must form the sequence
    \[
        (n+2, n+1, \ldots, \widehat{n-k+a}, \ldots, \widehat{n-k-1+b}, \ldots, n-k+1).
    \]
    Thus,
    \[
    \alpha_i = 
    \begin{cases}
        2  & \text{for } i=1,\ldots,k+2-a, \\
        1  & \text{for } i=k+3-a,\ldots,k+2-b, \\
        0  & \text{otherwise}.
    \end{cases}
    \]
    We conclude that $\alpha^T=(k+2-b, k+2-a)$. Finally,
    for such $\alpha$ all the elements in~\ref{eq:ogr-k-2}
    are as prescribed by the Borel--Bott--Weil theorem, and
    the number of inversions equals precisely $a+b-4$, which ends the proof.
\end{proof}

\begin{proof}[Proof of Proposition~\ref{prop:ogr-k-2}]
    Consider spectral sequence~\eqref{eq:ss} associated with filtration~\eqref{eq:olambda}
    on $\Lambda^{k+2}\CT$. Let us show that
    $V^{\langle 2\omega_{k+2} \rangle}\subset E^{k-2, 0}_\infty$.

    First, we compute the multiplicity of $V^{\langle 2\omega_{k+2} \rangle}$ in
    \[
        E^{k-2, 0}_2 = \bigoplus_{|\lambda| = k+2}H^{k-2}(X, \Sigma^\lambda(\CUpU)\otimes \Sigma^{\lambda^T}\CU^*).
    \]
    Let $\lambda$ be such that $|\lambda|=k+2$. Let $(\CUpU)^{\langle\beta\rangle}\subset\Sigma^\lambda(\CUpU)$
    be an irreducible summand such that
    $H^{k-2}(X, (\CUpU)^{\langle\beta\rangle}\otimes \Sigma^{\lambda^T}\CU^*) = V^{\langle 2\omega_{k+2} \rangle} = V^{\langle 2^{k+2}\rangle}$,
    where $2^h = (h, h)^T$.
    From the Borel--Bott--Weil theorem we see that $\left(\lambda^T\right)_1\leq 2$,
    so $\dht(\lambda)=\dwd(\lambda^T)\leq 2$.
    Again, from Borel--Bott--Weil one easily gets $|\beta|+|\lambda^T| = |2^{k+2}| = 2(k+2)$.
    Here, for $\beta\in\YD_k$ we still put $|\beta|=\sum_{i=1}^k \beta_i$ even if $\beta_k < 0$.
    From Lemma~\ref{lm:res}, we conclude that $\beta=\lambda$.
    In the same lemma we showed that $(\CUpU)^{\langle\lambda\rangle}$ appears in $\Sigma^\lambda(\CUpU)$
    with multiplicity $1$.
    Finally, assume $\lambda=(a, b)$, where $a+b=k+2$,
    is such that
    $H^{k-2}(X, (\CUpU)^{\langle\lambda\rangle}\otimes \Sigma^{\lambda^T}\CU^*) = V^{\langle 2\omega_{k+2} \rangle}$.
    By Lemma~\ref{lm:ogr-k-2}, the only restriction on $\lambda$ is that $b\geq 2$. Thus, the multiplicity of $V^{\langle 2\omega_{k+2} \rangle}$
    in $E^{k-2,0}_2$ equals the number of integer pairs $a\geq b\geq 2$ with $a+b=k+2$,
    which equals $\lfloor \frac{k}{2} \rfloor$.

    Next, we show that the sum of multiplicities of $V^{\langle 2\omega_{k+2} \rangle}$
    in $E^{k-3,q}_2$ for all $q>0$ is strictly smaller than $\lfloor \frac{k-1}{2} \rfloor$.
    Since the spectral sequence is equivariant, this will show that the multiplicity
    of $V^{\langle 2\omega_{k+2} \rangle}$ in $E^{k-2,0}_\infty$ is strictly positive.
    Let $q>0$. Then
    \[
        E^{k-3, q}_2 = \bigoplus_{\substack{|\lambda|=k+2-q,\\ \mu \in \DB_{2q}}}H^{k-3}(X, \Sigma^\lambda(\CUpU)\otimes \Sigma^{\lambda^T}\CU^*\otimes\Sigma^{\mu}\CU^*).
    \]
    Fix a pair $\lambda$, $|\lambda|=k+2-q$, and $\mu\in\DB_{2q}$.
    Let $(\CUpU)^{\langle\beta\rangle}\subset\Sigma^\lambda(\CUpU)$
    and $\Sigma^\alpha\CU^*\subset\Sigma^{\lambda^T}\CU^*\otimes\Sigma^{\mu}\CU^*$
    be irreducible summands such that
    $H^{k-3}(X, (\CUpU)^{\langle\beta\rangle}\otimes \Sigma^{\alpha}\CU^*) = V^{\langle 2\omega_{k+2} \rangle} = V^{\langle 2^{k+2}\rangle}$.
    A simple consequence of the Littlewood--Richardson rule is that
    $\alpha_1\leq \mu_1+\left(\lambda^T\right)_1$.
    From the Borel--Bott--Weil theorem we know that
    $\alpha_1\leq 2$, so $\dht(\lambda)=\dwd(\lambda^T)=\left(\lambda^T\right)_1\leq 2$,
    Again, from Borel--Bott--Weil $|\beta|+|\alpha| = |2^{k+2}| = 2(k+2)$.
    Since $|\alpha|=|\mu|+|\lambda^T| = 2q+(k+2-q)=k+q+2$ and
    $|\beta|\leq |\lambda| = k+2-q$ (see Lemma~\ref{lm:res}), we conclude that $|\beta|=|\lambda|$.
    By Lemma~\ref{lm:res}, the latter implies that $\beta=\lambda$.
    Recall that, by the same lemma, $(\CUpU)^{\langle\lambda\rangle}$ appears in $\Sigma^\lambda(\CUpU)$
    with multiplicity $1$.
    
    Assume that $\lambda=(a, b)$. By Lemma~\ref{lm:ogr-k-2},
    if $H^{k-3}(X, (\CUpU)^{\langle\lambda\rangle}\otimes \Sigma^{\alpha}\CU^*) = V^{\langle 2\omega_{k+2} \rangle}$,
    then $k\geq a\geq b\geq 2$, $k-3=a+b-4$, and $\alpha^T=(k+2-b, k+2-a)$. In particular, $a+b=|\lambda| = k+1$. Thus, $q=1$,
    and $\Sigma^\alpha\CU^*\subset \Sigma^{(a, b)^T}\CU^*\otimes \Lambda^2\CU^*$.
    Since $a+b=k+1$, $\alpha^T = (a+1, b+1)$. By Pieri's rule
    $\Sigma^{(a+1, b+1)^T}\CU^*$ appears with multiplicity $1$ in $\Sigma^{(a, b)^T}\CU^*\otimes \Lambda^2\CU^*$
    if $a > b$, and does not appear at all if $a=b$.
    Thus, $V^{\langle 2\omega_{k+2} \rangle}$ appears in $E^{k-3,1}_2$
    with multiplicity equal to the number of pairs $a > b \geq 2$ such that
    $a+b=k+1$, which is $\lfloor \frac{k-2}{2} \rfloor$, and does
    not appear in $E^{k-3,q}_2$ for $q\neq 0$.
    Since $\lfloor \frac{k-2}{2} \rfloor < \lfloor \frac{k}{2} \rfloor$,
    the statement follows.
\end{proof}

\printbibliography

\end{document}